\documentclass[10pt,letterpaper,dvips]{article}

 \usepackage{amssymb,latexsym,amsmath,amsthm}

\usepackage{fullpage}

\newtheorem{thm*}{Theorem}

\newtheorem{thm}{Theorem}

\newtheorem{lemma}{Lemma}

\newtheorem{remark}{Remark}

\newtheorem{prop}{Proposition}

\begin{document}

\def\W{ \widehat{X}  } 

\def\d{ \partial_{x_j} }
\def\Na{{\mathbb{N}}}

\def\Z{{\mathbb{Z}}}

\def\IR{{\mathbb{R}}}

\newcommand{\E}[0]{ \varepsilon}

\newcommand{\la}[0]{ \lambda}

\newcommand{\s}[0]{ \mathcal{S}}

\newcommand{\AO}[1]{\| #1 \| }

\newcommand{\BO}[2]{ \left( #1 , #2 \right) }

\newcommand{\CO}[2]{ \left\langle #1 , #2 \right\rangle}

\newcommand{\R}[0]{ \IR\cup \{\infty \} }

\newcommand{\co}[1]{ #1^{\prime}}

\newcommand{\p}[0]{ p^{\prime}}

\newcommand{\m}[1]{   \mathcal{ #1 }}


\newcommand{ \A}[1]{ \left\| #1 \right\|_H }

\newcommand{\B}[2]{ \left( #1 , #2 \right)_H }

\newcommand{\C}[2]{ \left\langle #1 , #2 \right\rangle_{  H^* , H } }

 \newcommand{\HON}[1]{ \| #1 \|_{ H^1} }

\newcommand{ \Om }{ \Omega}

\newcommand{ \pOm}{\partial \Omega}

\newcommand{\D}{ \mathcal{D} \left( \Omega \right)}

\newcommand{\DP}{ \mathcal{D}^{\prime} \left( \Omega \right)  }

\newcommand{\DPP}[2]{   \left\langle #1 , #2 \right\rangle_{  \mathcal{D}^{\prime}, \mathcal{D} }}

\newcommand{\PHH}[2]{    \left\langle #1 , #2 \right\rangle_{    \left(H^1 \right)^*  ,  H^1   }    }

\newcommand{\PHO}[2]{  \left\langle #1 , #2 \right\rangle_{  H^{-1}  , H_0^1  }}

 \newcommand{\HO}{ H^1 \left( \Omega \right)}

\newcommand{\HOO}{ H_0^1 \left( \Omega \right) }

\newcommand{\CC}{C_c^\infty\left(\Omega \right) }

\newcommand{\N}[1]{ \left\| #1\right\|_{ H_0^1  }  }

\newcommand{\IN}[2]{ \left(#1,#2\right)_{  H_0^1} }

\newcommand{\INI}[2]{ \left( #1 ,#2 \right)_ { H^1}}

\newcommand{\HH}{   H^1 \left( \Omega \right)^* }

\newcommand{\HL}{ H^{-1} \left( \Omega \right) }

\newcommand{\HS}[1]{ \| #1 \|_{H^*}}

\newcommand{\HSI}[2]{ \left( #1 , #2 \right)_{ H^*}}

\newcommand{\WO}{ W_0^{1,p}}
\newcommand{\w}[1]{ \| #1 \|_{W_0^{1,p}}}

\newcommand{\ww}{(W_0^{1,p})^*}

\newcommand{\Ov}{ \overline{\Omega}}

\author{C. Cowan\thanks{Department of Mathematics, University of Manitoba, Winnipeg, Manitoba, Canada R3T 2N2. Email: craig.cowan@umanitoba.ca. Research supported in part by NSERC.} }

\title{Supercritical elliptic problems involving  a Cordes like operator}
\maketitle

\begin{abstract}

In this work we obtain positive bounded solutions of various perturbations of 
\begin{equation} 
 \left\{ \begin{array}{lcl}
\hfill  -\Delta u - \gamma \sum_{i,j=1}^N \frac{x_i x_j}{|x|^2} u_{x_i x_j}  &=&  u^p \qquad \mbox{ in } B_1,   \\
\hfill  u &=& 0 \hfill \mbox{ on }   \partial B_1,
\end{array}\right.
  \end{equation}  where $B_1$ is the unit ball in $ \IR^N$ where $N \ge 3$, $ \gamma>0$ and     $ 1<p<p_{N,\gamma}$ where 
   \begin{equation*} 
 p_{N,\gamma}:=\left\{ \begin{array}{lc}
 \frac{N+2+3 \gamma}{N-2-\gamma} &  \qquad \mbox{ if } \gamma<N-2,    \\
\infty &  \qquad  \mbox{ if }   \gamma \ge N-2.
\end{array}\right.
  \end{equation*}   Note for $\gamma>0$ this allows for supercritical range of $p$.

\end{abstract}


\section{Introduction} 


In this work we are interested in obtaining positive bounded solutions of various perturbations of 
\begin{equation} \label{main_eq}
 \left\{ \begin{array}{lcl}
\hfill  -\Delta u - \gamma \sum_{i,j=1}^N \frac{x_i x_j}{|x|^2} u_{x_i x_j}  &=&  u^p \qquad \mbox{ in } B_1 \backslash \{0\},   \\
\hfill  u &=& 0 \hfill \mbox{ on }   \partial B_1,
\end{array}\right.
  \end{equation} 
  where $B_1$ is the unit ball in $ \IR^N$ where $N \ge 3$, $ \gamma>0$ and     $ 1<p<p_{N,\gamma}$  where  \begin{equation*} 
  p_{N,\gamma}:= 
  \left\{ 
  \begin{array}{cl}
   \frac{N+2+3 \gamma}{N-2-\gamma} & \quad \mbox{ if } \; \;  \gamma< N-2,\\
   \infty &  \quad \mbox{ if } \; \; \gamma \ge N-2.
   \end{array}
   \right.
   \end{equation*}   Note for $\gamma>0$  this includes a supercritical range of $p$, ie. $ p > \frac{N+2}{N-2}$.  The linear operator on the left hand side of (\ref{main_eq}) 
   is a known operator that has seen some investigation, see Section  \ref{cordes_sec} for more details.   The two main perturbations of (\ref{main_eq}) we consider are 
  \begin{equation} \label{eq_zero_pert}
 \left\{ \begin{array}{lcl}
\hfill  -\Delta u - \gamma \sum_{i,j=1}^N \frac{x_i x_j}{|x|^2} u_{x_i x_j}  &=& (1+\delta g(x)) u^p \qquad \mbox{ in } B_1 \backslash \{0\},   \\
\hfill  u &=& 0 \hfill \mbox{ on }   \partial B_1,
\end{array}\right.
  \end{equation} where $g$ is a fixed H\"older continuous function and $ \delta>0$ is a  small parameter;   and 
  
  \begin{equation} \label{eq_second_pert}
 \left\{ \begin{array}{lcl}
\hfill  -\Delta u - \gamma \sum_{i,j=1}^N \frac{x_i x_j}{|x|^2} u_{x_i x_j}  &=&  u^p \qquad \mbox{ in } \Omega_\delta \backslash \{0\},   \\
\hfill  u &=& 0 \hfill \mbox{ on }   \partial \Omega_\delta,
\end{array}\right.
  \end{equation} where $ \Omega_\delta$ is a small $C^2$ perturbation of $B_1$; see Section \ref{pert_111} for details.

  \begin{thm} \label{main_non_linear_zero} (Zero order perturbation)  Suppose $N \ge 3$ and $ 1<p<p_{N,\gamma}$ and $ g$ is a H\"older continuous function.  Then for sufficiently small $ \delta$ there is a positive solution $u \in C^{2,\alpha}_{loc}( \overline{B_1} \backslash \{0\}) \cap L^\infty$ of (\ref{eq_zero_pert}).
  \end{thm}

  \begin{thm} \label{main_non_linear_second} (Second order perturbation)  Suppose $N \ge 3$ and $ 1<p<p_{N,\gamma}$.  
  
  \begin{enumerate} \item  Suppose $ \gamma>N-2$. Then for sufficiently small $ \delta$ there is a positive solution $u \in C^{2,\alpha}_{loc}( \overline{\Omega_\delta} \backslash \{0\}) \cap L^\infty$ of (\ref{eq_second_pert}). 
  
  \item Suppose $0<\gamma <N-2$.  Then for sufficiently small $ \delta$ there is a nonnegative nonzero  solution $u \in C^{2,\alpha}_{loc}( \overline{\Omega_\delta} \backslash \{0\}) \cap L^\infty$ of (\ref{eq_second_pert}). 
   
  \end{enumerate} 
  \end{thm}
  
  \begin{remark}
  We are not addressing the exact smoothness of the solution at the origin and we are also stating the results on  punctured domains.   Since the solutions are bounded one can easily show these are suitable weak solutions on the full domain (and not just the punctured domains). 
  \end{remark}

  \subsection{A Cordes  like operator} \label{cordes_sec}

    For $ \gamma>0$ we define 
  \[L_\gamma(\phi)(x):=\Delta \phi(x)+ \gamma \sum_{i,j=1}^N \frac{x_i x_j}{|x|^2} \phi_{x_i x_j} = \Delta \phi(x)+ \gamma \phi_{rr}(x),\] where we are using spherical coordinates for the last term; ie.  $ x= r \theta$ where $ r=|x|$ and $ \theta:=\frac{x}{|x|} \in S^{N-1}$.       This explicit operator is often examined when one examines elliptic operators of the form 
  \[ \tilde{L}(\phi):=-\sum_{i,j=1}^N a_{ij}(x) \phi_{x_i x_j},\] where $a_{ij}$ are such that the operator is uniformly elliptic,  but the $a_{ij}$'s are not continuous.  
  One defines the \emph{Cordes Condition} by: there is some small $\E>0$ such that  \begin{equation}  \label{Cordes_cond}
  \frac{\left( \sum_{i=1}^N a_{i,i}(x) \right)^2 }{  \left( \sum_{i,j=1}^N a_{i,j}(x)^2 \right)} \ge N-1+\E, 
  \end{equation} then the operator $\tilde{L}:H^2(\Omega) \cap H_0^1(\Omega) \rightarrow L^2(\Omega)$ is an isomorphism (assuming $\Omega$ is bounded with smooth boundary)   see \cite{cordes_1, cordes_2, cordes_3, cordes_4,cordes_5}  for results related to the Cordes Condition.  Our operator is typically used to show the optimality of the Cordes Condition results. 
 If we consider our explicit example we see  if $ \E>0$ such that 
  \begin{equation}  \label{Cordes_cond_2}
 \frac{(N+\gamma)^2}{N+2\gamma+\gamma^2} \ge N-1+\E 
  \end{equation}  then $ L_\gamma$ satisfies (\ref{Cordes_cond}).  Checking the details one sees that if $ 0<\gamma<\frac{N}{N-2}$  then we can apply the above result to see $L_\gamma$ is an isomorphism.   There are results that extend this result to show that $\tilde{L}: W^{2,p} \cap W^{1,p}_0 \rightarrow L^p$ is an isomorpism for $p$ close to $2$ and there are also results for $L_\gamma$ on various spaces, including Morrey spaces.  Our function spaces will allow us to obtain results (which will be sufficient to apply our fixed point argument) regarding $L_\gamma$ for the full range of $\gamma>0$.

We now define the function spaces, which are motivated by 
\cite{MP}. Towards this define $ A_s:=\{x \in \IR^N: s<|x|<2s \}$ and for $\sigma \in \IR$ and $ N <t < \infty$ ($t$ is chosen larger than $N$ just to allow us to apply the Sobolev Imbedding Theorem and obtain pointwise gradient bounds) define  
the spaces $Y=Y_{t,\sigma}$ and $ X=X_{t,\sigma}$ with norms given by 
\[ \|f\|_Y^t:=\sup_{0<s \le \frac{1}{2}} s^{(2+\sigma)t-N} \int_{A_s} |f(x)|^t dx \] 
\[ \| \phi \|_X^t:= \sup_{0<s \le \frac{1}{2}} s^{\sigma t-N} \left\{ \int_{A_s} | \phi|^t dx + s^t \int_{A_s}  | \nabla \phi|^t dx  + s^{2t} \int_{A_s} |D^2 \phi|^t dx \right\} \] where for the space $X$ we impose the boundary condition $ \phi=0$ on $ \partial B_1$.  We now consider the linear problem given by     

\begin{equation} \label{linear_100}
 \left\{ \begin{array}{lcl}
\hfill   L_\gamma(\phi)   &=&   f(x)\qquad \mbox{ in } B_1 \backslash \{0\},   \\
\hfill  \phi &=& 0 \hfill \mbox{ on }   \partial B_1,
\end{array}\right.
  \end{equation} with goal of proving existence of solutions  with suitable estimates on $\phi$ in terms of $ f$.  When looking for solutions of (\ref{linear_100}) we will decompose into spherical harmonics and hence we need to consider the eigenpairs of the Laplace-Beltrami operator $ \Delta_\theta$ on $ S^{N-1}$.   For $ k \ge $ we have 
  \[ -\Delta_\theta \psi_k(\theta)  =\lambda_k \psi_k(\theta), \quad \theta \in S^{N-1},\]  and where we $L^2(S^{N-1})$ normalize $ \psi_k$.  Note that $ \lambda_0=0$ (multiplicity 1),  $ \lambda_1= N-1$ (multiplicity $N$) and $ \lambda_2 = 2 N$.  We now state our theorem related to the Cordes operator.

\begin{thm} \label{main_linear} (Cordes operator result) Suppose $N \ge 3$ and $ N<t<\infty$.

\begin{enumerate} \item Suppose   $ 0<\gamma <N-2$ and  $ 0<\sigma<  \frac{N-2-\gamma}{1+\gamma}$.  Then $L_\gamma$ an isomorphism from $X$ to $Y$.

\item Suppose $ \gamma> N-2$ and  $ \frac{N-2-\gamma}{1+\gamma} < \sigma <0$.  Then $L_\gamma:X \rightarrow Y$ is an isomorphism.

\item Suppose $0<\gamma < N-2$ and 
\begin{equation} \label{1mode}
\frac{N-2-\gamma}{2(1+\gamma)} - \frac{\sqrt{ (N-2-\gamma)^2+4(1+\gamma)(N-2)}}{2(1+\gamma)} <\sigma<0.  
\end{equation}
Then $L_\gamma:X_1 \rightarrow Y_1$ is an isomorphism where $ X_1$, $Y_1$ are the closed subspaces of $X,Y$ with no $k=0$ modes. 

\end{enumerate}
\end{thm}


\subsection{General background on the Lane-Emden equation} 


A well studied problem is the existence versus non-existence of positive solutions of the Lane-Emden equation given by
  \begin{eqnarray} \label{bound_p_d}
 \left\{ \begin{array}{lcl}
\hfill   -\Delta u   &=& u^p \qquad \mbox{ in } \Omega,  \\
\hfill u &=& 0 \qquad  \quad  \mbox{on } \pOm,
\end{array}\right.
  \end{eqnarray}
    where $\Omega$ is  a bounded domain in $ \IR^N$ with  $N \ge 3$. Define the critical exponent $p_s= \frac{N+2}{N-2}$ and note that it is related to the critical Sobolev imbedding exponent $ 2^*:=\frac{2N}{N-2} = p_s+1$.  For $ 1 < p < p_s$  $ H_0^1(\Omega) $ is compactly imbedded in $L^{p+1}(\Omega)$ and hence standard methods show the existence of a positive minimizer of
     \[ \min_{u \in H_0^1(\Omega) \backslash \{0\} } \frac{\int_\Omega | \nabla u|^2 dx }{  \left( \int_\Omega |u|^{p+1} dx \right)^\frac{2}{p+1}}.\] This positive minimizer is a positive solution of (\ref{bound_p_d}) see for instance the book  \cite{STRUWE}.
        For $ p \ge p_s$   $H_0^1(\Omega)$ is no longer compactly imbedded in $L^{p+1}(\Omega)$ and so to find positive solutions of (\ref{bound_p_d}) one needs to take other approachs.   For $ p \ge p_s$ the well known Pohozaev identity \cite{POHO} shows there are no positive solutions of (\ref{bound_p_d}) provided $ \Omega$ is star shaped.    For general domains in the critical/supercritical case, $ p \ge p_s$, the existence versus nonexistence of positive solutions of (\ref{bound_p_d}) is a very delicate question; see \cite{Coron,Passaseo,M_1,M_2,M_3}  and for related problems 
        \cite{addd_100, addd_101, addd_102, addd_103, addd_104}.  
        
There has been much work done on the existence and nonexistence of positive classical solutions of
\begin{equation} \label{lane_class}
-\Delta w = w^p  \qquad \mbox{in $ \IR^N$.}
\end{equation}  As in the bounded domain case the critical exponent $p_s$  plays a crucial role. For $ 1 <p < p_s$ there are no positive classical solutions of (\ref{lane_class}) and for $  p \ge  p_s$ there exist positive classical solutions,  see \cite{Caf,chen,gidas,Gidas}. The moving plane method shows that all positive classical solutions, satisfying certain assumptions, are radial about a point.


\subsection{Outline of the approach}

In Section \ref{radial_sec} we will construct a smooth positive radial solution $ w$ of (\ref{main_eq}) for $ 0<\gamma$ and $ 1<p<p_{N,\gamma}$.   Our approach to obtaining positive bounded solutions to (\ref{eq_zero_pert}) and (\ref{eq_second_pert}) will be to linearize around the radial solution $w$ (see Section \ref{radial_sec}) and hence of crucial importance will be the mapping properties of the linearized operator given by 
\[L(\phi):=\Delta \phi + \gamma \phi_{rr} + p w(r)^{p-1} \phi = L_\gamma(\phi) + p w(r)^{p-1} \phi.\]
 We look for solutions of (\ref{eq_zero_pert}) of the form $ u(x)=w(r)+ \phi(x)$ and then note we need $ \phi$ to satisfy (note we are replacing the term $u^p$ with $|u|^p$ in the equation and we will prove $u$ is positive at a later point; which is standard practice)
\begin{equation}  \label{non_lin_pert1}
 \left\{ \begin{array}{lcl}
\hfill -L(\phi) &=&  \delta g(x) |w+\phi|^p + |w+\phi|^p-w^p - p w^{p-1} \phi \qquad \mbox{ in } B_1 \backslash \{0\},   \\
\hfill  \phi &=& 0 \hfill \mbox{ on }   \partial B_1.
\end{array}\right.
  \end{equation}    We now define the nonlinear mapping $J_\delta$ via $ J_\delta(\phi)=\psi$ where $ \psi$ satisfies 
  \begin{equation}  \label{non_lin_pert1_map}
 \left\{ \begin{array}{lcl}
\hfill -L(\psi) &=&  \delta g(x) |w+\phi|^p + |w+\phi|^p-w^p - p w^{p-1} \phi \qquad \mbox{ in } B_1 \backslash \{0\},   \\
\hfill  \phi &=& 0 \hfill \mbox{ on }   \partial B_1.
\end{array}\right.
  \end{equation} 
  
To obtain a solution $u=w + \phi$ we will apply Banach's fixed point theorem to see that $J_\delta$ has a fixed point $ \phi$ for suitably small $ \delta$.   We will then argue $ u=w+\phi$ is positive and hence satisfies (\ref{eq_zero_pert}).
   See Section \ref{fixed_point_arg} for details of the fixed point argument and for (\ref{eq_second_pert}).

\section{A radial solution} \label{radial_sec}

In this section we construct a positive radial solution of (\ref{main_eq}) for a range of $p$ (which includes a supercritical range).

\begin{thm} \label{rad_sup} (Supercritical radial solutions) \begin{enumerate} 
\item  For $ 0<\gamma<N-2$ and $ 1<p<p_{N,\gamma}:=\frac{N+2+3 \gamma}{N-2-\gamma}$ there is a positive smooth radial decreasing solution $ w=w_\gamma$ of (\ref{main_eq}) in $ B_1$ in $ \IR^N$. 

\item For $ \gamma \ge N-2$ and $ 1<p<\infty$ there  is a positive smooth radial decreasing solution $ w=w_\gamma$ of (\ref{main_eq}) in $ B_1$ in $ \IR^N$.

\end{enumerate} 
\end{thm}

\begin{proof} Note $ w=w(r)$ is a solution of (\ref{main_eq}) in $B_1$ provided
  \[ - w''(r) - \frac{N-1}{r} w'(r) - \gamma w''(r) = w(r)^p, \quad 0<r<1\]  with $ w(1)=0$.  Note we can re-write this as 
  
  \[ -w_{rr} - \frac{N-1}{1+\gamma} \frac{w_r}{r} = \frac{ w^p }{1+\gamma} \] and note if we set $N_\gamma$ by 
\[  N_\gamma-1 = \frac{N-1}{1+\gamma} \]  then we can view the above problem as  
\begin{equation} \label{shit}
-\Delta_{N_\gamma} w = \frac{w^p}{1+\gamma} \qquad B_1 \subset \IR^{N_\gamma}
\end{equation}
with $ w=0$ on $ \partial B_1$ where $ \Delta_{N_\gamma}$ is the radial Laplacian in dimension $N_\gamma$.     

We now consider the two cases seperately.  Firstly we assume $ 0< \gamma < N-2$.  So if
\[ 1<p< \frac{N_\gamma+2}{N_\gamma-2} \] then the problem is subcritical and we can find a positive smooth radial solution.  A computation shows 
\[ \frac{N_\gamma+2}{N_\gamma -2} = \frac{N+2+3 \gamma}{N-2-\gamma}.\]  Note for $ \gamma>0$ this gives a supercritical range of $ p$.   We now consider the case of $ \gamma \ge N-2$.   Return to (\ref{shit}) and note $N_\gamma= \frac{N+\gamma}{1+\gamma}$.    A computation shows that $N_\gamma \le 2$ exactly when $ \gamma \ge N-2$ and hence we see (\ref{shit}) is subcritical in the case of $\gamma \ge N-2$ and hence we can find a positive smooth radial decreasing solution of (\ref{shit}) for any $ 1<p<\infty$.

\end{proof}

\subsubsection{Nondegeneracy of the radial solution.}

It is well known that the positive radial solution of the subcritical problem $ -\Delta u = u^p$ in $B_1$ with $u=0$ on $ \partial B_1$ is nondegenerate in the sense that the linearized operator $ \phi \mapsto \Delta \phi + p u(r)^{p-1} \phi$ has a trivial kernel in $H_0^1(B_1)$ (for instance); see \cite{Grossi_Pacella, Korman, Lin_100}.  This proof can be extended to show that the solution $w$, constructed in Theorem   \ref{rad_sup} is nondegenerate in $ H_{0,rad}^1(B_1 \subset \IR^{N_\gamma})$, of course it does not extend to show the full nondegeneracy of the solution $w$.

 We now state our kernel result.   The exact function space setting will vary depending on which situation we are in.   Essentially we want to cover all the cases from Theorem \ref{main_linear}.

\begin{prop} \label{kern_L}  Let $ p,\gamma,N$ be from the  hypothesis of Theorem \ref{rad_sup} and let $w$ be the smooth positive solution promised.  We now restrict $\sigma$ as in the various cases of Theorem \ref{main_linear}.  Set \begin{equation} \label{linear_L}
L(\phi):=\Delta \phi + \gamma \phi_{rr} + p w(r)^{p-1} \phi = L_\gamma(\phi) + p w(r)^{p-1} \phi.
\end{equation}   Suppose $ \phi \in X$ (or $ X_1$ as in the final case) such that $L(\phi)=0$ in $B_1 \backslash \{0\}$.  Then $ \phi=0$. 
\end{prop}

\begin{proof}  We write $ \phi(x)=\sum_{k=0}^\infty a_k(r) \psi_k(\theta)$ and then we have for all $k \ge 0$
\begin{equation} \label{ode_kern}
(1+\gamma) a_k''(r) + \frac{N-1}{r} a_k'(r) - \frac{\lambda_k a_k(r)}{r^2} + p w(r)^{p-1} a_k(r)=0 \quad 0<r<1,
\end{equation} with $ a_k(1)=0$.  
From the comments in the paragraph proceeding the theorem, we have  $ a_0=0$; the only possible issues are related to how singular $ a_0$ is.  We cover the case of $k=0$ later.  We now suppose $ k\ge 1$ and let $ v(r):=w_r(r)$ and note $ v<0$ and satisfies 
\[ 0=\Delta_{N_\gamma} v(r) + \frac{ p w(r)^{p-1} v}{1+\gamma} - \frac{(N-1)}{(1+\gamma)r^2} v \quad \mbox{ in } B_1 \backslash \{0\} \subset \IR^{N_\gamma} .\]  Also note that we can re-write the equation for $a_k$ as
\[ 0 = \Delta_{N_\gamma} a_k + \frac{p w(r)^{p-1} a_k}{1+\gamma}- \frac{\lambda_k a_k}{r^2(1+\gamma)} \quad B_1 \backslash \{0\} \subset \IR^{N_\gamma}, \] with $ a_k=0$ on $ \partial B_1$.  We now suppose $a_k \neq 0$ and we let $ T \in (0,1]$ denote the first positive $T$ such that $ a_k(T)=0$.    By multiplying $a_k$ by a constant we can assume $a_k>0$ in $B_T \subset \IR^{N_\gamma}$.    We now multiply the equation for $v$ by $ a_k$ and the equation for $a_k$ by $ v$ and integrate over $B_T \backslash B_\E \subset R^{N_\gamma}$ (where $ \E>0$ is small compared to $T$) to arrive at 
\begin{eqnarray*}
\frac{(\lambda_k -(N-1)}{\gamma+1} \int_{B_T \backslash B_\E} \frac{v a_k}{r^2} dx &= & a_k'(T) v(T) | \partial B_T|_{\IR^{N_\gamma}} + I_\E - J_\E
\end{eqnarray*} where $ | \partial B_T|_{\IR^{N_\gamma}}$ means the surface area  of the boundary of $B_T$ in $ \IR^{N_\gamma}$; ie. is equal to $ C_{N,\gamma} |T|^{N_\gamma-1}$ where $C_{N,\gamma}>0$ is a constant and where $ I_\E,J_\E$ are some surface integrals coming from the integration by parts.  These terms are equal to 
\[ I_\E=  C_{N,\gamma}v'(\E) a_k(\E) |\E|^{N_\gamma-1}, \quad J_\E=C_{N,\gamma}v(\E) a_k'(\E) |\E|^{N_\gamma-1}.
\]  Lets assume we can show that $I_\E,J_\E \rightarrow 0$ as $ \E \searrow 0$.     Then we would have 
\begin{equation} \label{kernel_100}\frac{(\lambda_k -(N-1)}{\gamma+1} \int_{B_T} \frac{v a_k}{r^2} dx = a_k'(T) v(T) | \partial B_T|_{\IR^{N_\gamma}}.
\end{equation} Now note that  $v<0$ in $B_T$ and $ \lambda_k-(N-1) \ge 0$ and hence the left hand side is less or equal zero.  By Hopf's Lemma we have $ a_k'(T)<0$ and hence the right hand side is positive;  this gives us the needed contradiction. 
Note that since $ \phi \in X$ one can show (here we are using assumption that $t>N$) to see that there is some $C>0$ such that $ |x|^{\sigma} | \phi(x)| + |x|^{\sigma+1}| \nabla \phi(x)| \le C$.    From this we see for each $k \ge 0$ we have $ |x|^\sigma |a_k(r)| + |x|^{\sigma+1} |a_k'(r)| \le C_k$ for all $ 0<|x| \le 1$.     One can easily show the following bounds on $ w(r)$;  $ |w'(r)| \le C r$ and $| w''(r)| \le C$. Using these estimates we see that 
\[ |I_\E| + |J_\E| \le C \E^{N_\gamma-1-\sigma},\] and hence we have the desired provided $ N_\gamma-1-\sigma=\frac{N-1}{1+\gamma}-\sigma>0$. \\

We now consider the various cases.  In the first case we have $0<\gamma<N-2$ and   $ 0<\sigma< \frac{N-2-\gamma}{1+\gamma}$ and hence we have $N_\gamma-1-\sigma>0$.   We now consider the second case where $ \gamma>N-2$ and $ \frac{N-2-\gamma}{1+\gamma}<\sigma<0$.  Note in this case that since $\sigma<0$ we trivially have the desired result.    The final case follows the same idea as case 2 since $ \sigma$ is negative.    \\

We now consider the case of $k=0$.   Here we follow the approach of \cite{Grossi_Pacella, Lin_100, Korman}.  Set $ \delta(r):=r w'(r)$ which is negative for $0<r \le 1$. A computation shows that 
\[ -\Delta_{N_\gamma} \delta(r) = \frac{p w(r)^{p-1} \delta(r)}{1+\gamma} + \frac{2 w(r)^p}{1+\gamma}, \quad \mbox{ in }   B_1 \backslash \{0\} \subset \IR^{N_\gamma}.\]    Multiply this equation by $a_0$ (which, towards a contradiction, we are assuming is not identically zero) and integrate over $ \{x: \E<|x|<1\}$ and use integration by parts and the equation for $a_0$ to arrive  at \
\[ \frac{2}{1+\gamma} \int_{ \{\E<|x|<1\} } w^p a_0 = \int_{\partial B_1} \delta \partial_\nu a_0 + I_\E - J_\E,\]  where 
\[ I_\E:= \int_{\partial B_\E} a_0 \partial_\nu \delta, \quad J_\E:=\int_{\partial B_\E} \delta \partial_\nu a_0,\] where (as above) we are in the possibly fractional dimension $N_\gamma$.   Lets assume $I_\E,J_\E \rightarrow 0$ as $ \E \searrow 0$.  Then we have 
\[ \frac{2}{1+\gamma} \int_{ B_1} w^p a_0 = \int_{\partial B_1} \delta \partial_\nu a_0,\] and by Hopf's lemma we have $ \partial_\nu a_0 = C \neq 0$ on $ \partial B_1$ and hence we have $ \int_{B_1} w^p a_0 \neq 0$.   By multiplying the equation for $a_0$ by $w$ and the equation for $w$ by $a_0$ (and taking a bit of care near the origin) we arrive at $  \int_{B_1} w^p a_0=0$; which gives us the desired contradiction.    To show $ I_\E,J_\E \rightarrow 0$  one using essentially the same argument as for $k \ge 1$. 
 
\end{proof}

\section{The linear theory}

\subsection{The Cordes operator $L_\gamma$}

\begin{lemma} \label{first_lemma}  Under the hypothesis of Theorem \ref{main_linear} the kernel of $L_\gamma$ is trivial and for all $ k \ge 0$ there is some $C_k>0$  such that for all $ f(x)= b_k(r) \psi_k(\theta)$ there is some $ \phi(x)=a_k(r) \psi_k(\theta)$ such that  $ \phi,f$ solve (\ref{linear_100}) and $ \| \phi \|_X \le C_k \|f\|_Y$.   The above results hold for case 1 and case 2.  For case 3 the result holds for all $k \ge 1$. 

\end{lemma}

\begin{proof} Firstly its clear that $L_\gamma:X \rightarrow Y$ is continuous and into $Y$ in both case 1 and case 2.  So we begin by showing the kernel of $L_\gamma$ is trivial. Suppose $ \phi(x)=\sum_{k=0}^\infty a_k(r) \psi_k(\theta)$ is in the kernel.  Then we have 
\[ (1+\gamma) a_k''(r) + \frac{(N-1) a_k'(r)}{r} - \frac{ \lambda_k a_k(r)}{r^2}=0 \quad 0<r<1\] with $ a_k(1)=0$.    Noting the equation is of Euler type we see the solutions are given via $ a_k(r)= C_k ( r^{\beta_k^+}- r^{\beta_k^-})$ where $ \beta_k^{\pm}$ is defined by the roots of 
\[ (1+\gamma) \beta^2 + (N-2-\gamma) \beta -\lambda_k=0,\] and hence is given by 
\[ \beta_k^{\pm}:= \frac{-(N-2-\gamma)}{2(1+\gamma)} \pm  \frac{ \sqrt{ (N-2-\gamma)^2 + 4 (1+\gamma) \lambda_k}}{2(1+\gamma)}.\] 

In both case 1 and 2 note that if $ \beta_k^- < -\sigma$ then $ C_k (r^{\beta_k^+}-r^{\beta_k^-}) $ is not an element of $X$ unless $ C_k=0$.   By monotonicity in $k$ it is sufficient that $ \beta_0^- < -\sigma$.  Note in case 1 this is exactly the condition that  $ 0<\sigma<\frac{N-2-\gamma}{1+\gamma}$.  In case 2 we want $ \beta_0^-<-\sigma$ and this is just the condition that $ \sigma<0$.   Further restrictions on $ \sigma$ will come later.

    We now prove the desired onto estimate for each mode $k \ge 0$.   
For each $k \ge 0$ consider 
\[ (\gamma+1) a_k''(r) + \frac{(N-1)  a_k'(r)}{r} - \frac{\lambda_k a_k(r)}{r^2} = b_k(r) \quad 0<r<1\] with $ a_k(1)=0$.  
Using the variation of parameters method we obtain solutions of the form 
 
\[ (\gamma+1)(\beta_k^-- \beta_k^+) a_k(r) = r^{\beta_k^-} \int_{T_2}^r \frac{ b_k(\tau)}{\tau^{\beta_k^- -1}} d\tau - r^{\beta_k^+} \int_{T_1}^r \frac{b_k(\tau)}{\tau^{\beta_k^+-1}} d\tau + C_k r^{\beta_k^+}+D_k r^{\beta_k^-},
\] where $C_k,D_k$ are free parameters (depending on $k$ and $b_k$) and we are free to choose $ T_i$ suitably; we need to pick these parameters such that we get the desired estimate on $a_k$ and such that $a_k(1)=0$. 
We will choose $ T_2=0$, $ T_1=1$, $ D_k=0$ and we leave  $C_k=C_k(b_k)$ free for now and hence we get 

\[ (\gamma+1)(\beta_k^-- \beta_k^+) a_k(r) = r^{\beta_k^-} \int_{0}^r \frac{ b_k(\tau)}{\tau^{\beta_k^- -1}} d\tau - r^{\beta_k^+} \int_{1}^r \frac{b_k(\tau)}{\tau^{\beta_k^+-1}} d\tau + C_k r^{\beta_k^+},
\] and note this is an acceptable  choice of $ T_2$ provided  $\frac{b_k(t)}{t^{\beta_k^--1}} \in L^1(0,1)$, which we assume for now. 
For simplicity we normalize $ \|b_k \psi_k \|_Y \le 1$ and hence there is some $\tilde{C}_k$ such that 
\begin{equation} \label{est_111} \int_s^{2s} |b_k(\tau)|^t d \tau \le \tilde{C}_k s^{1-t(2+\sigma)} \quad 0<s \le \frac{1}{2}.
\end{equation} 
We now prove that $ \frac{b_k(\tau)}{\tau^{\beta_k^--1}} \in L^1(0,1)$.

\begin{eqnarray*}
\int_0^1 \frac{ |b_k(\tau)|}{ \tau^{\beta_k^- -1}} d \tau & \le  & C_k \sum_{i=0}^\infty  2^{i(\beta_k^--1)} \int_{2^{-i-1}}^{2^{-i}}  |b_k(\tau)|  d \tau \\
& \le & C_{k,1}  \sum_{i=0}^\infty  2^{i(\beta_k^--1)} \left(\int_{2^{-i-1}}^{2^{-i}}  |b_k(\tau)|^t  d \tau \right)^\frac{1}{t}  2^{\frac{-i}{t'}} \\
& \le & C_{k,1}  \sum_{i=0}^\infty  2^{i(\beta_k^--1)} 2^{i(2+\sigma-\frac{1}{t}) } 2^{\frac{-i}{t'}} \\
&=& C_{k,1} \sum_{i=0}^\infty  2^{i ( \beta_k^--1+2 + \sigma - \frac{1}{t} - \frac{1}{t'})}
\end{eqnarray*} and note the exponent is simplifies to $ \beta_k^- + \sigma$.   So provided $ \beta_k^- + \sigma<0$  then the sum converges and we get the desired result.  By the mononocity in $ k$ its sufficient to consider the case of $k=0$,  ie.  we want $ \beta_0^-+\sigma<0$.    We first consider case 1 and in this case this restriction is exactly the assumption that $ 0<\sigma< \frac{N-2-\gamma}{1+\gamma}$.   In case 2 we have $\beta_0^-=0$ and hence the restriction just becomes that $ \sigma <0$ (in case 2 there will be further restrictions on $\sigma$ later).     We now consider the various terms in the formula for $a_k$.  \\

We first examine the term 
\[r^{\beta_k^-} \int_{0}^r \frac{ b_k(\tau)}{\tau^{\beta_k^- -1}} d\tau  + C_k r^{\beta_k^+}\] and we choose 
\[ C_k:=- \int_0^1 \frac{ b_k(\tau)}{\tau^{\beta_k^- -1}} d\tau.\]   Note with this choice of $C_k$ we have the needed zero boundary condition for this term (and its clear the other term has the needed boundary condition)  hence $ a_k(1)=0$.   We now get the estimate.    Firstly we will need the term $ r^{\beta_k^+} \in X$.  In case 1 this will require that $\beta_k^+ \ge -\sigma$ and by monotonicy in $k$ its sufficient that $\beta_0^+ \ge -\sigma$  but this holds since $\beta_0^+=0$ and $ \sigma>0$.   In case 2 we again will need $ \beta_0^+ \ge -\sigma$ and writing this out gives $ 0> \sigma \ge \frac{N-2-\gamma}{1+\gamma}$.  So in both cases we have $ r^{\beta_k^+} \in X$. 
Now note by the previous argument to show the needed integrand is $L^1(0,1)$ we have $ |C_k| $ is bounded by a constant depending just on $k$ and hence in both case 1 and 2 we have $ \| C_k r^{\beta_k^+} \psi_k \|_X$ is bounded by a constant just depending on $k$.  We now need to examine the integral term and the computation is very similar to when showing the previous integrand was $L^1(0,1)$.   A computation shows 

\begin{eqnarray*}
\int_0^r \frac{ |b_k(\tau)|}{\tau^{\beta_k^--1}} d \tau & \le C_k &\sum_{i=0}^\infty (r 2^{-i})^{1-\beta_k^-}  \int_{r 2^{-i-1}}^{r 2^{-i}} | b_k(\tau)| d \tau \\
& \le & C_k \tilde{C}_k \sum_{i=0}^\infty (r 2^{-i})^{1-\beta_k^- + \frac{1}{t'}} \left(  \int_{r 2^{-i-1}}^{r 2^{-i}} | b_k(\tau)|^t d \tau \right)^\frac{1}{t}\\
& \le & C_{k,1} \sum_{i=0}^\infty (r 2^{-i})^{1-\beta_k^- + \frac{1}{t'} + \frac{1}{t}-2-\sigma} \\
&=& r^{-\beta_k^--\sigma} C_{k,1} \sum_{i=0}^\infty \frac{1}{  (2^{-\beta_k^--\sigma})^i}. 
\end{eqnarray*} 
Note in both case 1 and 2 we have $ -\beta_k^--\sigma>0$ and hence the infinite sum converges.   From this we see in either case we have 
\[ r^{\beta_k^-+\sigma} \int_0^r \frac{ |b_k(\tau)|}{\tau^{\beta_k^--1}} d \tau \le D_k\] and this gives us the needed zero order estimate on one of the integral terms.  \\

We now consider the other integral term namely 
\[r^{\beta_k^+} \int_{1}^r \frac{b_k(\tau)}{\tau^{\beta_k^+-1}} d\tau=: r^{\beta_k^+} g_k(r).\] Note that  we can write (for integers $n \ge 1$)
\[ g_k(2^{-n})= \sum_{i=1}^n ( g_k(2^{-i})-g_k(2^{-i+1})) \; \; \mbox{ and hence } \; \;  | g_k(2^{-n})| \le  \sum_{i=1}^n | g_k(2^{-i})-g_k(2^{-i+1})|.  \]   A computation similar to the previous one shows 
\begin{eqnarray*}
| g_k(2^{-n})| & \le & \sum_{i=0}^n \int_{ 2^{-i}}^{2^{1-i}} \frac{ |b_k(\tau)|}{\tau^{\beta_k^+-1}} d \tau \\
& \le & C_k \sum_{i=0}^n 2^{i(\beta_k^+ +\sigma)} \\
&=& C_k \frac{2^{ (\beta_k^++\sigma)(n+1)-1}}{2^{\beta_k^+ +\sigma}-1}
\end{eqnarray*} and from this we see 
\[  ( 2^{-n})^{\beta_k^+ + \sigma}| g_k(2^{-n})| \le \tilde{D}_k\] for all $n \ge 1$.   This gives us the desired zero order estimate at least for the values of $ r \in \{ 2^{-n}: n \ge 1 \mbox{ an integer} \}$.    One can extend the above estimate for all values of $r$ and hence combining all the above results gives us the needed zero order estimate on $a_k(r)$.   The higher order portions of the norm of $a_k$ can be obtained from the zero order estimates after consider the equation that $a_k$ satisfies.

In case 3 everything works as in the previous two cases except now one just needs $ 0<\beta_1^++\sigma$ and $ \beta_1^-+\sigma<0$. 
\end{proof}

\noindent 
\textbf{Proof of Theorem \ref{main_linear}.}   In case 1 or case 2, by Lemma \ref{first_lemma}, for all $ k \ge 0$ there is some $C_k$ such that for all $ f(x)= b_k(r) \psi_k(\theta)$ there is some $ \phi(x)= a_k(r) \psi_k(\theta)$ which solves (\ref{linear_100}) and $ \| \phi \|_X \le C_k \|f\|_Y$.    One can show for all $ m \ge 1$ there is some $D_m$ such that one has for all $ f(x)=\sum_{k=0}^m b_k(r) \psi_k(\theta)$ there is some $ \phi(x)=\sum_{k=0}^m a_k(r) \psi_k(\theta)$  which solves (\ref{linear_100}) and $ \| \phi \|_X \le D_m \|f\|_Y$.     We now will show that $D_m$ is bounded.  Suppose not,  then there is some $ f_m \in Y$ and $ \phi_m \in X$ which solve (\ref{linear_100}) and $ \|f_m\|_Y \rightarrow 0$ and $ \| \phi_m \|_X=1$.    We claim that 
\[ \sup_{0<s<\frac{1}{2}} s^{\sigma t-N} \int_{A_s} | \phi_m|^t dx \rightarrow 0.\]   Towards a contradiction we assume, after passing to a subsequence,  that this quantity is bounded below by $ 2\E_0>0$ and hence there is some $ 0<s_m<\frac{1}{2}$ such that 
\[s_m^{\sigma t-N} \int_{A_{s_m}} | \phi_m|^t dx \ge \E_0.\]   We consider two cases: \\ 
Case (i); $ s_m$ bounded away from zero, and after passing to a subsequence we can assume $s_m \rightarrow s \in (0,\frac{1}{2}]$.\\ 
Case (ii); $ s_m \rightarrow 0$. \\

\noindent 
Case (i).    Since $ \| \phi_m\|_X \le 1$ we see that $ \phi_m$ is bounded in $ W^{2,t}_{loc}( \overline{B_1} \backslash \{0\})$ and after passing to a subsequence we have $ \phi_m \rightharpoonup \phi $ in $ W^{2,t}_{loc}( \overline{B_1} \backslash \{0\})$ and one can use weak lower semi continuity of the norms to see that $ \phi \in X$.  Also we have 
\[ \E_0 \le s_m^{\sigma t-N} \left( \int_{A_{s_m} \Delta A_s} | \phi_m|^t dx + \int_{A_s} | \phi_m|^t dx \right),\] where $A \Delta B:= A \backslash B \cup B \backslash A$ is the symmetric difference of $A$ and $B$.   Note that $ \phi_m$ bounded in $L^\infty_{loc}( \overline{B_1} \backslash \{0\})$ and $|A_{s_m} \Delta A_s| \rightarrow 0$.  Using this we can pass to the limit to see 
\[ \E_0 \le s^{\sigma t -N} \int_{A_s} | \phi|^t dx,\] and hence $ \phi \in X$ is non-zero.   Note also we can pass to the limit in the equation to see that $L_\gamma(\phi)=0 $ in $B_1 \backslash \{0\}$ with $ \phi=0$ on $ \partial B_1$  but this contradicts the result from Lemma \ref{first_lemma} which says the kernel of $L_\gamma$ is trivial.  \\

\noindent
Case (ii). Set $ \zeta_m(z):= s_m^\sigma \phi_m(s_m z)$ defined on $ 0<|z|<\frac{1}{s_m}$. For $ i \ge 2$ an integer we set $ E_i:=\{x \in \IR^N:  \frac{1}{i}<|x|<i\}$ and $\tilde{E}_i:= \{ x \in \IR^N:\frac{1}{2i}<|x|<2i \}$ and note that  
\begin{equation} \label{est_before}
\int_{1<|z|<2} | \zeta_m(z)|^t dz \ge \E_0, \qquad 
 \int_{\tau <|x|<2 \tau} | \zeta_m(z)|^t dz \le \tau^{N-\sigma t},
 \end{equation} for all $ 0<\tau \le \frac{1}{2 s_m}$. 

Note that $ \zeta_m(z)$ satisfies 
\begin{equation} \label{limit_1} 
L_\gamma(\zeta_m)(z) =g_m(z):=s_m^{\sigma+2} f_m(s_m z) \quad \mbox{ in } \; 0<|z|< \frac{1}{s_m},
\end{equation}  with $ \zeta_m=0$ on $|z|= \frac{1}{s_m}$.   Note that for each fixed $i$ we have $ \|g_m\|_{L^t(\tilde{E}_i)} \rightarrow 0$.  Also note the equation is satisfied on  $ \tilde{E}_i$ for all $i$ and for sufficiently large $m$.   By elliptic regularity and the estimates in (\ref{est_before}) we see that $ \zeta_m$ is bounded in $ W^{2,t}(E_i)$ for large enough $m$ and hence by a diagonal argument we can assume there is some $\zeta$ such that $ \zeta_m \rightharpoonup \zeta$ in $ W^{2,t}_{loc}( \IR^N \backslash \{0\})$ and $ \zeta$ satisfies both estimates in (\ref{est_before}) (and hence $\zeta \neq 0$).    Moreover we have $L_\gamma(\zeta)=0$ in $ \IR^N \backslash \{0\}$.    We now obtain the needed contradiction which amounts to showing the kernel of $L_\gamma$ is trivial over the appropriate space.  We write $ \zeta(z):= \sum_{k=0}^\infty a_k(r) \psi_k(\theta)$ and as usual $a_k$ will be of the form 
\[ a_k(r)= C_k r^{\beta_k^+} + D_k r^{\beta_k^-},\] here we are omitting writing out the individual ode's for each mode since we have already done this on the unit ball.  We will now translate the second estimate  in (\ref{est_before}) to some estimates on $a_k$. Note for each $ k \ge 0$ there is some $ \hat{C}_k>0$ such that 
\[ a_k(r)= \hat{C}_k \int_{| \theta|=1} \zeta(r \theta) d \theta,\] and then by Jensen's inequality 
\[ \int_\tau^{2 \tau} r^{N-1} |a_k(r)|^t dr \le \tilde{D}_k \int_\tau^{2 \tau} r^{N-1} \int_{| \theta|=1} | \zeta(r \theta)|^t d \theta dr \le \tilde{D}_k \tau^{N-\sigma t}\] for all $ \tau>0$.  Putting the explicit form of $a_k(r)$ in to the integral and using a change of variables we arrive at 
\[ \int_1^2 s^{N-1} \big| C_k \tau^{\beta_k^++\sigma} s^{\beta_k^+} + D_k \tau^{\beta_k^-+\sigma} s^{\beta_k^-} \big|^t ds \le \tilde{D}_k \] for all $ \tau>0$.  Note that  $\beta_k^+ +\sigma \neq \beta_k^-+\sigma$ and provided both are nonzero we can send $ \tau$ to $0$ or $ \infty$ to obtain a contradiction unless  $C_k=D_k=0$. Note we have both of these exponents are nonzero and hence we have that $\zeta=0$ a contradiction.   \\

\noindent
Case 3. In this case  everything follows as in the previous cases except now one needs $ \beta_1^-+\sigma<0$ and $ \beta_1^+ +\sigma>0$. 
  \hfill $\Box$

\subsection{The linearized operator $L$} 

Here we examine the linearized operator 
\[ L(\phi)(x)=L_\gamma(\phi) + p w(r)^{p-1} \phi = \Delta \phi + \gamma \phi_{rr} + p w(r)^{p-1} \phi.\]   In this section we consider the solvability of

\begin{equation} \label{linear_L}
 \left\{ \begin{array}{lcl}
\hfill   L(\phi)   &=&   f(x)\qquad \mbox{ in } B_1 \backslash \{0\},   \\
\hfill  \phi &=& 0 \hfill \mbox{ on }   \partial B_1,
\end{array}\right.
  \end{equation}

  \begin{thm} \label{linear_L_space}  \begin{enumerate} 
   \item Under the assumption of Theorem \ref{main_linear} part 1 there is some $C>0$ such that for all $ f \in Y$ there is some $ \phi \in X$ which solves (\ref{linear_L}) and $ \| \phi \|_X \le C \|f\|_Y$. 
  
      \item  Under the assumption of Theorem \ref{main_linear} part 2 there is some $C>0$ such that for all $ f \in Y$ there is some $ \phi \in X$ which solves (\ref{linear_L}) and $ \| \phi \|_X \le C \|f\|_Y$. 
      
      \item Under the assumptions of Theorem \ref{main_linear} part 3 there is some $C>0$ such that for all $ f \in Y_1$ there is some $ \phi \in X_1$ which solves (\ref{linear_L}) and $ \| \phi \|_X \le C \|f\|_Y$. 
      
      \item Let $0<\gamma<N-2$. There is some $ C>0$ such that for all bounded  $ f=f(r)$ there is some $ \phi=\phi(r)$ which solves (\ref{linear_L}) such that 
      \[ \sup_{0<r<1} \left( |\phi(r)|+ | \phi'(r)| \right) \le C \sup_{0<r<1} |f(r)|.\] 
      
  \end{enumerate}

  \end{thm}

  \begin{proof} 1 and 2. Define $K:X \rightarrow Y$ by $ K(\phi):= p w^{p-1} \phi$.  It is easily seen that $K$ is a compact mapping from $X$ to $Y$ and note we can write $L=L_\gamma +K$.   So we have the desired result via Fredholm theory provided the only $ \phi \in X$ such that $L(\phi)=0$ is $ \phi=0$.   But this follows from Proposition  \ref{kern_L}.  \\
  
  \noindent
  3.  This follows exactly the same proof as part 1 and 2 of the current theorem, we just need to check that $K:X_1 \rightarrow Y_1$ is compact.   \\ 
  
  \noindent
  4.  Define $K_0(g)=u$ where $ (\gamma+1) u''(r) + \frac{N-1}{r} u'(r) = g(r)$ in $ 0<r<1$ with $ u(1)=0$.   We get an explicit formula for $K_0$.  Given $ g$ define 
  \[ h(r):=\frac{1}{r^\frac{N-1}{\gamma+1}} \int_0^r  \frac{ \tau^\frac{N-1}{\gamma+1} g(\tau)}{\gamma+1} d \tau,\] and we define $u(r)$ via $ -u(r):=\int_r^1 h(\hat{t}) d \hat{t}$ and then note $u$ satisfies the required ode and we have the estimate 
  \[ \sup_{0<r<1} \left( |u(r)| + |u'(r)| \right) \le C_3 \sup_{0<r<1} |g(r)|.\]  This shows that the mapping  $K_0:L^\infty_{rad}(B_1) \rightarrow L^\infty_{rad}(B_1)$ is compact.    We now try and solve (\ref{linear_L}) and we will use the notation $ \phi(r)=a_0(r)$ and $ f(r)=b_0(r)$. Then note to solve 
  $L(a_0)=b_0(r)$ in $0<r<1$ with $ a_0(1)=0$ we can write this is as $ a_0 + K_0(p w^{p-1} a_0)= K_0(b_0)$ and if the only $ a_0 \in L^\infty_{rad}$ such that $a_0 + K_0(p w^{p-1} a_0)=0$ is $ a_0=0$  then by Fredholm theory there is some $C_0>0$ such that $ \sup_{0<r<1} |a_0(r)| \le C_0 \sup_{0<r<1} |K_0(b_0)|$ and it will be clear that $\sup_{0<r<1} |K_0(b_0)| \le C_1 \sup_{0<r<1} |b_0(r)|$ and hence we'd have $ \sup_{0<r<1} |a_0(r)| \le C_2 \sup_{0<r<1} |b_0(r)|$.   Now recalling Proposition \ref{kern_L} we have the desired kernel is empty.    We now return to $ a_0 = K_0(b_0 - p w^{p-1} a_0)$ and then note by the earlier estimate this gives 
  \[ \sup_{0<r<1} \left( |a_0(r)|+ |a_0'(r)| \right) \le C_3 \sup_{0<r<1} \big| b_0 - p w^{p-1} a_0 \big| \le C_4 \sup_{0<r<1} | b_0|.\] 
  
 \end{proof}

  \section{The fixed point arguments} \label{fixed_point_arg} 
  
  \subsubsection{Equation (\ref{eq_zero_pert})}
  
Here we obtain a positive bounded solution $u$ of (\ref{eq_zero_pert}) on $B_1 \backslash \{0\}$ and recall we are looking for solutions of the form $ u=w+ \phi$ where $ \phi$ solves (\ref{non_lin_pert1}).  To prove the existence of $ \phi$ we will show that the nonlinear mapping $J_\delta$ (as defined by $ J_\delta(\phi)=\psi$ where $ \psi$ satisfies (\ref{non_lin_pert1_map})) is a contraction on a suitable space.   In the process of doing this we will need the following facts: for $p>1$ there is some $ C_p>0$ such that for all $ 0<w \in \IR$ and $ \phi, \hat{\phi} \in \IR$  
\begin{equation}  \label{first_app}
\big|  |w+\phi|^p - p w^{p-1} \phi - w^p \big| \le C_p \left( w^{p-2} \phi^2 + | \phi|^p \right) 
\end{equation} 
\begin{equation}  \label{sec_app}
\big|  |w+ \hat\phi|^p - |w+\phi|^p - p w^{p-1} ( \hat \phi - \phi) \big| \le C_p \left( w^{p-2} ( | \phi| + | \hat \phi|) + | \phi|^{p-1} + | \hat \phi|^{p-1} \right) | \hat \phi - \phi|.
\end{equation} 

The exact spaces we will work on will depend on the value of $ \gamma$; we split this into the cases $ \gamma>N-2$ and $ 0 <\gamma <N-2$.  The first case will be the easy case and is fairly standard and we work directly in $X$ (recall $X$ depends on $ \sigma$ and $t$).   For the second case we could do the same but the issue now is $ \phi$ can then be unbounded near the origin which would force $u$ to be unbounded near the origin and recall we want $u$ bounded.  One could try and apply elliptic regularity but we prefer to avoid this since we are dealing with a nonstandard operator with possible issues at the origin.  Additionally we want $u=w+ \phi$ to be positive and hence to show this we either need $ \phi$ small in $L^\infty$ (with an additional argument near the boundary)  or we can instead try and apply maximum principles to show $u$ positive.  We will use the first approach and so this causes us to use slightly more complicated function space.  \\

\noindent
\textbf{Case 1.}  $ \gamma>N-2$.   In this case  we fix $ N<t< \infty$ and $ \sigma$ as in  Theorem \ref{main_linear} part 2.  We now show that $J_\delta:X \rightarrow X$ (here $X$ is defined as before with $ t$ and $ \sigma$ as above). Note by a scaling argument and the Sobolev imbedding there is some $C_1>0$ such that for all $ \zeta \in X$ we have 
\begin{equation} \label{1011}
\sup_{A_s} | \zeta | \le \frac{C_1 \| \zeta \|_X}{s^\sigma},
\end{equation} for all  $0<s \le \frac{1}{2}$. 
Let $C>0$ be from  Theorem \ref{linear_L_space} part 2. \\

\noindent \textbf{Into.} Let $ \phi \in B_R \subset X$ where $0<R \le 1$ (here $B_R$ is the closed ball of radius $R$ centered at the origin in $X$) and let $ \psi=J_\delta(\phi)$   (we are attempting to show that $ J_\delta$ is into $B_R$). Then we have 
\[ \| \psi \|_X \le C |\delta| \| g |w+ \phi|^p \|_Y + C \| |w+\phi|^p - w^p - p w^{p-1} \phi \|_Y,\] note since $Y$ is basically an $L^t$ norm we can replace the desired term with the upper bound coming from (\ref{first_app}).  It is easily seen that there is some $C_2>0$ (independent of $ 0<R \le 1$) we have  $C |\delta| \| g |w+ \phi|^p \|_Y \le C_2 | \delta|$.     A direct computation shows that provided $ \sigma \le \frac{2}{p-1}$ we have $ \| | \phi |^p \|_Y \le C_2 R^p$ but note we have $ \sigma<0$ and hence this estimate holds.   We now examine the term $ \| w^{p-2} \phi \|_Y$.   Note that for $p<2$  there are some added difficulties for this term near the boundary of $B_1$.   Note that using the above argument we have 
\[ \sup_{0<s \le \frac{1}{4}} s^{(2+\sigma)t-N} \int_{A_s} w^{p-2} | \phi|^{2t}dx \le C_2 R^{2t}\] provided $ \sigma \le 2$, which again trivially holds since $ \sigma<0$.   We now examine the portion of the norm for $ s $ close to $ \frac{1}{2}$ where for $ p<2$  the term $ w^{p-2}$ can cause problems.  Using a scaling argument and the Sobolev imbedding we obtain the existence of some $ C_2>0$  such that $ | \phi(x)| \le C_2 R \delta(x)$ for all $ \frac{1}{2} \le |x| \le 1$, where $ \delta(x):=dist(x, \partial B_1)$ is the Euclidean distance from $x$ to $ \partial B_1$.  Using this estimate we see that 
\begin{equation} \label{near_one}
\sup_{\frac{1}{4} \le |x| <1 } w^{(p-2)t} | \phi|^{2t} \le \sup_{\frac{1}{4} \le |x| <1 }  w(x)^{pt} \left( \frac{ C_2 R \delta(x)}{w(x)} \right)^{2t} \le C_3 R^{2t},
\end{equation} since  $ w(x) \ge \E \delta(x)$ on $B_1$ for some $ \E>0$ small enough.  From this and the earlier estimate we can conclude $ \| w^{p-2} \phi^2 \|_Y \le C_2 R^2$.  Combining the estimates shows that $ \| \psi \|_X \le C_4 | \delta| + C_4 R^2 + C_4 R^p$ and hence for $J_\delta$ to be into $B_R$ it is sufficient that 
\begin{equation} \label{into_0}
C_2 \left( | \delta| + R^2 + R^p \right) \le R.
\end{equation}

\noindent
\textbf{Contraction.} Let $ \hat \phi,\phi \in B_R$ and $ J_\delta( \hat \phi) =\hat \psi,  J_\delta(\phi)= \psi$. Then we have 
\[ -L( \hat{\psi}- \psi)= \delta g \left( | w+ \hat \phi|^p - |w+ \phi|^p \right) + | w + \hat \phi|^p - |w+ \phi|^p - p w^{p-1} ( \hat \phi - \phi).\]  Using (\ref{sec_app}) we see 
\begin{eqnarray*}
\frac{ \| \big| \hat \psi - \psi \|_X}{C_p} &\le & C |\delta| \Big\| \big| w^{p-2} ( | \hat \phi| + | \phi|) + | \hat \phi |^{p-1} + | \phi |^{p-1} + p w^{p-1} \big| | \hat \phi - \phi| \Big\|_Y  \\
&& + \Big\| \big| w^{p-2} ( | \hat \phi| + | \phi|) + | \hat \phi |^{p-1} + | \phi |^{p-1}  \big| | \hat \phi - \phi| \Big\|_Y
\end{eqnarray*} and note the first term on the right differs  from the second by only the linear term $ p w^{p-1} | \hat \phi - \phi |$ and hence we can drop the first term on the right by taking $ \delta>0$ small.  Writing out the estimate $ \| | \phi |^{p-1} | \hat \phi - \phi \|_Y$ we see provided $ 2 -\sigma(p-1) \ge 0$  then we have this term is bounded above by $ C R^{p-1} \| \hat \phi - \phi \|_X$.  We now examine the term $ \| w^{p-2} | \phi| | \hat \phi - \phi | \|_Y$.  A computation shows that 
\[ \| w^{p-2} | \phi| | \hat \phi - \phi| \|_Y^t \le \sup_{0<s < \frac{1}{2}} \left( s^{2t} \sup_{A_s} w^{(p-2)t} | \phi|^t \right) \| \hat \phi - \phi \|_X^t.\]  As before the case of $p<2$ causes an added issue for $ s \nearrow \frac{1}{2}$.  Using (\ref{near_one}) we see $ \sup_{\frac{1}{4}<s < \frac{1}{2}}  s^{2t} \sup_{A_s} w^{(p-2)t} | \phi|^t \le C_4 R^t$ and hence we need to just obtain an estimate for $ 0<s \le \frac{1}{4}$.  Note that a computation shows $ \sup_{0<s < \frac{1}{4}}  s^{2t} \sup_{A_s} w^{(p-2)t} | \phi|^t \le C_4 R^t$ provided $ 2-\sigma \ge 0$ which we have.   Combining all these results we arrive at:  by fixing $ 0<R$ sufficiently small and then taking $ | \delta$ sufficiently small we see that we have $ \| \hat \psi - \psi \|_X \le K_0 \| \hat \phi - \phi \|_X$ where $K_0<1$.   Moreover by fixing $ 0<R$ small and then taking $ | \delta|$ small we see we can satisfy (\ref{into_0}) and hence $J_\delta$ is a contraction on $B_R$.   By applying Banach's fixed point theorem we see $J_\delta$ has a fixed point $ \phi$ and hence $u=w + \phi$ solves (\ref{eq_zero_pert}) in $B_1 \backslash \{0\}$ but with $ u^p$ replaced with $|u|^p$.   By taking into account the function spaces we see that $u$ is bounded and by taking $ R>0$ small we see that $u$ is positive provided we stay away from $ \partial B_1$.    By using the fact we have an estimate like $ \sup_{\frac{1}{2}<|x|<1} | \nabla \phi| \le C_5 R$ and since $ w'(1)<0$ we see that by taking $ R$ small that we have $ u=w+ \phi>0$ in $B_1$.  \\

\noindent 
\textbf{Case 2.} $ 0<\gamma < N-2$.  Let $ N<t<\infty$ and $ \sigma$ as in Theorem 3 part 3.   We now define the space we work in.  Given $ \phi(x)$ we write 
\[ \phi(x):=\sum_{k=0}^\infty a_k(r) \psi_k(\theta) = a_0(r) + \sum_{k=1}^\infty a_k(r) \psi_k(\theta)=: \phi_0(r) + \phi_1(x).\]   Define the  $ \W$ norm of $ \phi$ via 
\[ \| \phi \|_{\W} = \sup_{0<r<1} \left\{ | \phi_0(r)| + | \phi_0'(r) | \right\} + \| \phi_1 \|_X, \] and we impose the boundary condition $ \phi=0$ on $ \partial B_1$.  Let $B_R$ denote the closed ball or radius $R$ centered at the origin in $ \W$. \\ 

\noindent
\textbf{Into.}  Let $ \phi \in B_R$ and $ \psi:=J_\delta(\phi)$.  Then note there is some $C>0$ (independent of $ 0<R \le 1$) such that  $ \| \phi \|_{L^\infty} \le C \|\phi \|_{\W}$ and so $ | \phi(x)| \le C R$.    Let $f$ denote the right hand side of (\ref{non_lin_pert1_map}) and note we have 
$ |f(x)| \le C \left( | \delta| + w^{p-2} \phi^2 + R^p \right) $ and in the case of $ p<2$ an additional argument shows that 
$ |f(x)| \le C \left( | \delta| +  R^2 + R^p \right).$   We write $ f(x)= f_0(r) + f_1(x)$ where we are using the notation introduced to decompose $ \phi(x)= \phi_0(r) + \phi_1(x)$ and then one sees that $ |f_0(r)| \le C \left( | \delta| +  R^2 + R^p \right).$    From this an the earlier ODE arguments we see that 
\[ \sup_{0<r<1} ( | \psi_0(r)| + | \psi_0'(r)|) \le C\left( | \delta| +  R^2 + R^p \right),\] since   $ -L(\psi_0)= f_0$.  Now note that $ -L(\psi_1)= f_1(x) = f(x)-f_0(r)$ and using the estimates on $ f(x)$ and $ f_0(r)$ we see that $ |f_1(x)| \le C \left( | \delta| +  R^2 + R^p \right).$  But using the linear theory for $L$ we have $ \| \psi_1 \|_X \le C \|f_1 \|_Y$ and note that provided $ 2+ \sigma \ge 0$ we can translate the $L^\infty$ bound on $f_1$ to a $Y$ bound on $ f_1$.  So provided $ \sigma \ge -2$ we have $\| \psi_1 \|_X \le C \|f_1 \|_Y \le C\left( | \delta| +  R^2 + R^p \right) $.  So note for $ \psi= J_\delta(\phi) \in B_R \subset \W$ it is sufficient that 
\begin{equation} \label{into_200}
2C\left( | \delta| +  R^2 + R^p \right) \le R.
\end{equation}

\noindent
\textbf{Contraction.} Let $ \hat \phi, \phi \in B_R \subset \W$ and $ \hat \psi:=J_\delta(\hat \phi)$ and $ \psi:= J_\delta(\phi)$.   Then note we have '
\[ -L(\hat \psi - \psi)= \delta g \left( |w+ \hat \phi|^p - | w + \phi |^p \right) + \left\{ | w + \hat \phi|^p - | w + \phi|^p - p w^{p-1} ( \hat \phi - \phi)\right\} =: F^1 + F^2=F.\]  By  (\ref{sec_app}) we have 
\[ | F^2| \le C \big| w^{p-2} ( | \hat \phi| + |  \phi | ) + | \phi|^{p-1} + | \hat \phi|^{p-1} \big| | \hat \phi - \phi|,  \quad \mbox{ and } \]
\[ | F^1| \le C | \delta| |g| C \big| w^{p-2} ( | \hat \phi| + |  \phi | ) + | \phi|^{p-1} + | \hat \phi|^{p-1} + p w^{p-1} \big| | \hat \phi - \phi|.\]   Note that from earlier arguments we have $ \sup_{B_1} | \phi| \le C R$. From this and an additional argument in the case of $ p<2$ (which we have already done) we see that  
\begin{equation} \label{1010}
|F(x)| \le C ( R + R^{p-1} + \delta)  | \hat \phi(x) - \phi(x)|.
\end{equation}
Using this and the earlier ode results we see 
\[ \sup_{0<r<1} \left\{ | \hat \psi_0(r) - \psi_0(r)| +  | \hat \psi_0'(r) - \psi_0'(r)|\right\} \le  C ( R + R^{p-1} + \delta) \sup_{0<r<1} \int_{| \theta |=1} | \hat \phi(r\theta) - \phi(r \theta)| d \theta,\]   and using the fact $ \| \zeta\|_{L^\infty} \le C \| \zeta \|_{\W}$ we arrive at  
\[ \sup_{0<r<1} \left\{ | \hat \psi_0(r) - \psi_0(r)| +  | \hat \psi_0'(r) - \psi_0'(r)|\right\} \le  C ( R + R^{p-1} + \delta)  \| \hat \phi - \phi \|_{\W}.\] 

Now note that $-L( (\hat \psi - \psi)_1) = F_1$ and hence $ \| (\hat \psi - \psi)_1 \|_X \le C \|F_1 \|_Y$.  Note that from the earlier computations we have 
\[ | F(x)|, |F_0(r)| \le C (R + R^{p-1} +\delta) \| \hat \phi - \phi \|_{\W},\] and hence we have the same pointwise bound for $F_1$.  This shows that for $ \sigma \ge -2$ we have $ \|F_1 \|_Y \le C (R + R^{p-1} +\delta) \| \hat \phi - \phi \|_{\W}$ and hence we have 
\[ \| ( \hat \psi - \psi )_1 \|_X \le C (R + R^{p-1} +\delta) \| \hat \phi - \phi \|_{\W}.\]    Combining this with the earlier result we have 
\[ \| \hat{\psi}- \psi \|_{\W} \le 2C (R + R^{p-1} +\delta) \| \hat \phi - \phi \|_{\W},\] and hence we see $J_\delta$ a contraction on $B_R \subset \W$ provided we have $2C (R + R^{p-1} +\delta)<1$ and (\ref{into_200}) holds.    Fix $ 0<R$ very small and then take $ |\delta|$ sufficiently small and we easily satisfy the two conditions.  By taking  $R>0$ small and using the bound on the gradient of $ \phi$ near $ \partial B_1$ (and the fact that $ w'(1)<0$) we see $ u=w+\phi>0$ in $B_1$.

  \subsubsection{Equation (\ref{eq_second_pert})} \label{pert_111}

In this section we want to prove the existence of positive solutions of (\ref{eq_second_pert}) which rewrite in terms of $ y \in \Omega_\delta$;
\begin{equation} \label{eq_second_pert_y}
 \left\{ \begin{array}{lcl}
\hfill  -\Delta_y u(y) - \gamma \sum_{i,j=1}^N \frac{y_i y_j}{|y|^2} u_{y_i y_j}(y)  &=&  |u(y)|^p \qquad  y \in \Omega_\delta \backslash \{0\},   \\
\hfill u &=& 0 \hfill  y \in   \partial \Omega_\delta,
\end{array}\right.
  \end{equation} where $\Omega_\delta$ is a small perturbation of the unit ball in $ \IR^N$ (and note we replaced  $u^p$ with $|u|^p$).

We now perform a change of variables to reduce the problem to one on the unit ball (we take this change of variables from \cite{pert_ori}). 
 Fix $\psi:\overline{B_1}\to \IR^N$ (and for simplicity of notation we assume $ \psi(0)=0$; otherwise $u$ would be singular at $ y_\delta =\delta \psi(0)$) be a smooth map and for ${\delta}>0$ define
\[
\Omega_{\delta}:= \{x +\delta\psi(x) : x\in B_1\}.
\]  This domain will be the small perturbation of the unit ball we work on. 
There is some small $\delta_0>0$ such that for all $0 <\delta<\delta_0$ one has that $\Omega_\delta$ is diffeomorphic to the unit ball $B_1$. Let $y=x+\delta \psi(x)$ for $x\in B_1$ and note there is some $\tilde{\psi}$ smooth such that $x=y+{\delta}\tilde{\psi}({\delta},y)$ for $y\in \Omega_{\delta}$. Given $u(y)$ defined on $y\in \Omega_{\delta}$ or $v(x)$ defined on $x\in B_1$ we define the other via
$u(y) =v(x)$.  So to find a positive singular solution $u(y)$ of (\ref{eq_second_pert_y}) it is sufficient to find a positive singular solution $v(x)$ of some, to be determined equation,  on the unit ball.    To compute the equation for $v(x)$ we will use the chain rule, but we mention that the computation becomes somewhat messy.  A computation shows that 

\[
\begin{array}{rl}
u_{y_iy_j}
=&v_{x_ix_j}+\delta\sum_{l=1}^Nv_{x_ix_l} {\tilde{\psi}^l}_{y_j}
+\delta\sum_{k=1}^Nv_{x_kx_j} {\tilde{\psi}^k}_{y_i}+\delta^2
\sum_{k=1}^Nv_{x_kx_j} {\tilde{\psi}^j}_{y_j}
{\tilde{\psi}^k}_{y_i}
\\ & \\
&\ \ \ \ \ \ \ \ +\delta^2
\sum_{k,h=1}^Nv_{x_kx_h} {\tilde{\psi}^h}_{y_j}
{\tilde{\psi}^k}_{y_i}
+\delta\sum_{k=1}^N v_{x_k}{\tilde{\psi}^k}_{y_iy_j}
\end{array}
\] and using this formula we can write $ \Delta_y u(y)=\Delta_x v(x) + E_\delta(v)$  and $ u_{y_i y_j} = v_{x_i x_j} + E_\delta^{i,j}(v)$.  Also we have
\[ \frac{y_i y_j}{|y|^2}= \frac{x_i x_j + \delta (x_i  \psi^j + \psi^i x_j) + \delta^2 | \psi|^2 }{ |x|^2+ 2 \delta x \cdot \psi + \delta^2 | \psi|^2}.\]   So $u(y)$ solves (\ref{eq_second_pert_y}) if $v(x)$ solves 
\begin{equation} \label{v_final}
0=\Delta v + E_\delta (v) + \gamma \sum_{i,j=1}^N \frac{ y_i y_j}{|y|^2} \left( v_{x_i x_j} + E_\delta^{i,j}(v) \right) + |v|^p \qquad \mbox{ in $B_1 \backslash \{0\}$,} 
\end{equation} 
with $v=0$ on $ \partial B_1$.  We look for solutions of the form $ v(x)= w(x)+ \phi(x)=w(r) + \phi(x)$.  A computation shows that $ \phi$ must satisfy 
\begin{eqnarray} \label{second_perturb_phi}
-L(\phi) &=& | w+ \phi|^p - w^p - p w^{p-1} \phi + E_\delta(w) + E_\delta(\phi)  \nonumber\\
&& + \gamma \sum_{i,j=1}^N \left( \frac{ y_i y_j}{|y|^2} - \frac{x_i x_j}{|x|^2} \right) \phi_{x_i x_j} \nonumber \\
&& + \gamma \sum_{i,j=1}^N \frac{ y_i y_j}{|y|^2} \left( E^{i,j}_\delta (w) + E_\delta^{i,j}( \phi) \right) \nonumber \\ 
&& + \gamma \sum_{i,j=1}^N \left( \frac{ y_i y_j}{|y|^2} - \frac{x_i x_j}{|x|^2} \right) w_{x_i x_j} \quad \mbox{ in } B_1 \backslash \{0\}, 
\end{eqnarray}  with $ \phi=0$ on $ \partial B_1$.  Note we replaced the $v^p$ term with $|v|^p$,  which is standard practice and one then later shows $v>0$.   Note all terms on the right hand side, except $| w+ \phi|^p - w^p - p w^{p-1} \phi$, are perturbation terms which are zero when $ \delta=0$.   As before we hope to find a solution of the above via a fixed point argument.  Towards this define $J_\delta(\phi)=\psi$ where 
\begin{eqnarray} \label{second_psi}
-L(\psi) &=& | w+ \phi|^p - w^p - p w^{p-1} \phi + E_\delta(w) + E_\delta(\phi)  \nonumber\\
&& + \gamma \sum_{i,j=1}^N \left( \frac{ y_i y_j}{|y|^2} - \frac{x_i x_j}{|x|^2} \right) \phi_{x_i x_j} \nonumber \\
&& + \gamma \sum_{i,j=1}^N \frac{ y_i y_j}{|y|^2} \left( E^{i,j}_\delta (w) + E_\delta^{i,j}( \phi) \right) \nonumber \\ 
&& + \gamma \sum_{i,j=1}^N \left( \frac{ y_i y_j}{|y|^2} - \frac{x_i x_j}{|x|^2} \right) w_{x_i x_j} \quad \mbox{ in } B_1 \backslash \{0\}.
\end{eqnarray}

\noindent
\textbf{Proof of Theorem \ref{main_non_linear_second}.}  Let $X,Y$ denote the spaces as defined before where either we are taking $ \sigma$ positive or negative along with the extra assumptions on $ \sigma$.  Then one can easily see that $J_\delta$ is a contraction on $B_R$ (closed ball of radius $R$ centered at the origin in $X$) provided $ 0<R$ is fixed small and then $\delta>0$ is chosen small enough.      Also by taking $ R$ sufficiently small we see that $ v=w+ \phi$ is not indentically zero.    \\

\noindent
1. In the case of $ \gamma>N-2$ (and hence $ \sigma<0$) for any $ \E>0$ we can take $ R>0$ small enough such that $ v=w+\phi$ is positive and bounded away from zero on $ |x|<1-\E$.   Also note that via the Sobolev imbedding we can make the gradient of $ \phi$ small away from the origin.  Since $ w'(1)<0$ we see this forces $ v>0$ near $ |x|=1$.     Hence in the case of $\gamma>N-2$ we can find a positive bounded solution $v$.  \\

\noindent
2. In this case we can still follow the above procedure to obtain a solution $ v(x)= w(x) + \phi(x)$ of (\ref{v_final}).  The relevant linear theory we are using is Theorem \ref{linear_L_space} part 1 and we will later take $ \sigma>0$ sufficiently small. 
Note that by taking $ R>0$ small we have $ v>0$ away from the origin.   But we might still have $ v$ change sign near the origin and also it might blow up at the origin.  So we have a nonzero solution $ v$ of (\ref{v_final}) and hence $u$ is a nonzero solution of  (\ref{eq_second_pert_y}).  Also note we have $u \in W^{2,t}_{loc}(\overline{\Omega_\delta} \backslash \{0\})$ with  the bounds $ |u(y)| |y|^\sigma + | \nabla u(y)| |y|^{\sigma+1} \le C$ for all $ y \in \Omega_\delta \backslash \{0\}$; after considering the bounds on $v$ and using the Sobolev Imbedding Theorem.   Also one can see $u$ inherits the same regularity of $v$ near the origin; ie for $ s>0$ small we have 
\[ s^{(2+\sigma) t-N} \int_{s<|y|<2s} |D^2 u(y)|^t dy \le C_1.\]    
 If we let $ f(y):= - |u(y)|^p $  then $ f \in L^T(\Omega_\delta)$ for all $T<\frac{N}{\sigma p}$ after consider the bound on $u$.  
 So rewriting the equation for $u$ in terms of $f$ we get
\begin{equation} \label{eq_second_reg}
 \left\{ \begin{array}{lcl}
\hfill L_\gamma(u)(y)= \Delta_y u(y) + \gamma \sum_{i,j=1}^N \frac{y_i y_j}{|y|^2} u_{y_i y_j}(y)  &=&  f(y) \qquad  \mbox{ in } \Omega_\delta \backslash \{0\},   \\
\hfill u &=& 0 \hfill  \mbox{ on }   \partial \Omega_\delta.
\end{array}\right.
\end{equation} 
 Now recall $X$ was a space of functions defined on the punctured unit ball with certain regularity assumptions.  We let $ \tilde{X}:=\{ u:  \exists v \in X \mbox{ with } u(y)=v(x) \}$.   So note that $ u \in \tilde{X}$.  We now prove a maximum principle and then return to the proof of part 2 of the theorem. 
 
 \begin{lemma} (Maximum Principle) Suppose $ u \in \tilde{X}$  with $ -L_\gamma(u)(y)=-f(y) \ge 0$ in $\Omega_\delta \backslash \{0\}$ (where $f$ is sufficiently regular away from the origin and has slight blow up at the origin; here we are modelling $f$ on the explicit $f$ above).  Then $ u \ge 0$ in $ \Omega_\delta \backslash \{0\}$. 
 \end{lemma} Note in the above lemma that $u$ is arbitrary, but of course we will apply the lemma for our specific $u$. 

\begin{proof} For $ \E>0 $ small we set $ u_\E$ to be a solution of 
\begin{equation} \label{eq_second_reg_eps}
 \left\{ \begin{array}{lcl}
\hfill -L_\gamma(u_\E)(y)=   &=& - f(y) \qquad  \mbox{ in } \Omega_{\delta,\E}:=\Omega_\delta \backslash B_\E,   \\
\hfill u &=& 0 \hfill  \mbox{ on }   \partial \Omega_{\delta,\E}.
\end{array}\right.
\end{equation}  By the maximum principle we have $u_\E \ge 0$ in $\Omega_{\delta,\E}$.   Under the assumption that $ \sigma<\frac{N-2-\gamma}{1+\gamma}$ we can use a maximum principle argument to see that 
\[ 0 \le |y|^\sigma u_\E(y) \le \frac{\sup_{z \in \Omega_\delta \backslash \{0\}}  |z|^{2+\sigma} |f(z)|}{ \sigma (N-1-(\sigma+1)(1+\gamma)},\] for all $ y \in \Omega_{\delta,\E}$.  Using a scaling argument along with the equation satisfied by $u_\E$ we can show that $u_\E$ satisfies second order weighted $L^t$ estimates on $ \Omega_{\delta,\E}$ similar to the estimates that $v$ satisfies on $B_1 \backslash \{0\}$.   Using a diagonal argument and passing to a subsequence one can show there exists some $\tilde{u}$ such that $u_\E \rightharpoonup \tilde{u}$ in $ W^{2,t}_{loc}( \overline{\Omega_\delta} \backslash \{0\})$ and 
$L_\gamma(\tilde{u})(y)=  f(y)$ in $ \Omega_\delta \backslash \{0\}$ with $ \tilde{u}=0$ on $ \partial \Omega_\delta$.  Moreover $ \tilde{u}$ satisfies the same weighted $L^t$ estimates near the origin as $u$.   From this we can conclude that $ \tilde{u} \in \tilde{X}$.   Hence if we can show the kernel of $L_\gamma$ is trivial on $\tilde{X}$ then we'd have $ \tilde{u}=u$ and hence $u$ is nonnegative.     We now transform variables to the unit ball.  Hence its sufficient to show the kernel of $ L_{\gamma,\delta}$ is trivial in $X$  where 
\[ L_{\gamma,\delta}(v):=L_\gamma(v)+ E_\delta(v)+ \gamma \sum_{i,j=1}^N \frac{y_i y_j}{|y|^2} E_\delta^{i,j}(v).\]  It is easily seen that once $\sigma$ is fixed  that for $ \delta>0$ small that the kernel of $L_{\gamma,\delta}$ is trivial.     This completes the proof of the maximum principle. 

\end{proof}

We now let $u$ denote the solution of the nonlinear problem as above.   From the above lemma we have $u  \ge 0$.   Our goal is to now show that $u$ is bounded on $ \Omega_\delta$.  Away from the origin its clear $u$ is bounded.    We assume that $ \delta>0$ is small enough such that $ B_\frac{3}{4} \subset \subset \Omega_{\delta}$.  
 Set $ U(y):=u(y) \phi(y)$ where $ 0 \le \phi \le 1$ is a smooth cut off with $ \phi=1$ in $ B_\frac{1}{4}$ and $ \phi \in C_c^\infty(B_\frac{1}{2})$.  Then a computation shows that 
\begin{equation} \label{GG}
L_\gamma(U)=g \mbox{ in } B_1 \backslash \{0\}, \qquad U=0 \mbox{ on } \partial B_1,
\end{equation} where 
\begin{eqnarray*}
g(y) &=& f(y) \phi(y) + 2 \nabla u \cdot \nabla \phi + u \Delta \phi \\
 &&+ \sum_{i,j=1}^N \frac{y_i y_j}{|y|^2} \left\{ u_{y_i} \phi_{y_j} + u_{y_j} \phi_{y_i} + u \phi_{y_i y_j} \right\}.
 \end{eqnarray*}   Note $ g$ is as smooth near the origin as $ f$ is and recalling the pointwise bound on $u$ near the origin gives $ | f(y)| |y|^{\sigma p} \le C$ on $ \Omega_\delta \backslash \{0\}$ and hence $ |y|^{\sigma p} |g(y)| \le C$ on $B_1 \backslash \{0\}$.   For notational convenience we rename the variable $y$ by $x$ since we are on the unit ball;  but note we are not using the change of variables.   As before we write in spherical harmonics as 
$U(x)=\sum_{k=1}^\infty a_k(r) \psi_k(\theta)$ and we decompose $U$ and $g$ as before.  So we write $U(x)=U_0(r)+U_1(x)$ and $ g(x)=g_0(r)+ g_1(x)$  where $ U_1,g_1$ have no $k=0$ modes.     Then we have $L_\gamma(U_i)(x)=g_i(x)$ in $B_1 \backslash \{0\}$ with $ U_i=0$ on $ \partial B_1$.      Fix $ \sigma_1<0$  but sufficiently close to zero such that $ \sigma_1$ satisfies (\ref{1mode}) (where we are replacing $ \sigma$ with $ \sigma_1$).    By taking $ | \sigma_1|$ smaller we can assume $ \sigma_1 + 2 -\sigma p >0$ and hence  $ |x|^{\sigma_1+2} |g_1(x)| \le C_1$ for all $ 0<|x|<1$ (a standard argument shows that $ g_i$ satisfies the same point wise estimates as $ g(x)$).   This shows that $g_1 \in Y_1$ (with respect to $\sigma_1$ see Theorem \ref{main_linear} part 3).  So we can now apply Theorem \ref{main_linear} part 3 to see that $ U_1 \in X_1$ (again with respect to $\sigma_1$) and hence $U_1 $ is bounded.   To complete the proof we need to show that $U_0$ is bounded.      Consider the proof of 
Theorem  \ref{linear_L_space} part 4 where we obtain an explicit solution for an ode.    Using this we can get an explicit formula for a solution of the equation for $U_0$.  To see this formula really gives $U_0$ we note that $L_\gamma$ is an isomorphism between spaces $X$ and $Y$ (and does not interact between different modes).     So from this we see 
\[ -U_0(R)= \int_R^1 h(r) dr \] where 
\[ h(r):= \frac{1}{r^\frac{N-1}{\gamma+1}} \int_0^r   \frac{\tau^\frac{N-1}{\gamma+1} g_0(\tau)}{\gamma+1}  d \tau. \]  Using the bound $ |g_0(\tau)| t^{\sigma p} \le C$ we see by taking $ \sigma>0$ small enough that we have $U_0$ bounded and this completes the proof. 

\hfill $\Box$

{}


\begin{thebibliography}{}


























 \bibitem{Caf} L. Caffarelli, B. Gidas and J. Spruck. \emph{Asymptotic symmetry and local behaviour of semilinear elliptic equations with critical Sobolev growth}. Commun. Pure Appl. Math. 42 (1989), 271–297.


\bibitem{chen} W. Chen and C. Li, \emph{Classification of solutions of some nonlinear elliptic equations}. Duke Math. J. 63 (1991), 615–622.

\bibitem{cordes_5} M. Chicco, \emph{Equazioni ellittiche del secondo ordine di tipo Cordes con termini di ordine inferiore}, Ann. Mat. Pura Appl. 85, 347-356 (1970)


        
        \bibitem{addd_102} M. Clapp, M. Grossi and A. Pistoia,  \emph{Multiple solutions to the Bahri-Coron problem in domains with a shrinking hole of positive dimension}, Complex Var. and Elliptic Eqns., 57, 1147-1162. 


\bibitem{cordes_2} 
HO. Cordes, \emph{Zero order a priori estimates for solutions of elliptic differential equations}, Proc. Symp. Pure Math. 4, 157-166 (1961)



    \bibitem{Coron}  J.M. Coron,  \emph{ Topologie et cas limite des injections de Sobolev}. C.R. Acad. Sc. Paris, 299, Series I, 209–212.(1984).








\bibitem{Grossi_Pacella} L. Damascelli, M. Grossi and F. Pacella, \emph{Qualitative properties of positive solutions of semilinear elliptic equations in symmetric domains via the maximum principle}, Annales de l'Institut Henri Poincare (C) Non Linear Analysis
Volume 16, Issue 5, September–October 1999, Pages 631-652.
























\bibitem{pert_ori} J. D\'avila and L. Dupaigne, Perturbing singular solutions of the Gelfand
problem. Commun. Contemp. Math. 9 (2007), no. 5, 639-680.



 







        \bibitem{addd_100} M. del Pino, \emph{Supercritical elliptic problems from a perturbation viewpoint},  Discrete and Continuous Dynamical Systems 21 (1), 69, 2008. 


\bibitem{M_2} M. del Pino, P. Felmer and  Monica Musso, \emph{Two-bubble solutions in the super-critical Bahri-Coron's problem,} Calculus of Variations and Partial Differential Equations 16 (2003), no. 2, 113-145 PDF

\bibitem{M_3} M. del Pino, P. Felmer and M. Musso, \emph{Multi-bubble solutions for slightly super-critical elliptic problems in domains with symmetries,} Bull. London Math. Society 35 (2003), no. 4, 513-521



\bibitem{M_1} M. del Pino and  M. Musso, \emph{Super-critical bubbling in elliptic boundary value problems,} Variational problems and related topics (Kyoto, 2002).  1307 (2003), 85-108.















   







\bibitem{Gidas}  B. Gidas, W. Ni and L. Nirenberg, \emph{Symmetry and related properties via the maximum principle}. Commun. Math. Phys. 68 (1979), 525–598. MR0544879 (80h:35043)




\bibitem{gidas}B. Gidas and J. Spruck, \emph{Global and local behavior of positive solutions of nonlinear elliptic equations.} Comm. Pure Appl. Math., 34(4):525-598, 1981.







        
        \bibitem{addd_103} F. Gladiali and M. Grossi, \emph{Supercritical elliptic problem with nonautonomous nonlinearities}, J. Diff. Eqns., 253 (2012), 2616-2645. 
        
        
        \bibitem{addd_101} M. Grossi and F. Takahashi, \emph{Nonexistence of multi-bubble solutions to some elliptic equations on convex domains}, Jour. Funct. Anal., 259 (2010), 904-917. 
        









\bibitem{Korman} P. Korman \emph{Global solution curves for semilinear elliptic equations}, World Scientific Publishing Co. Pte. Ltd.2012. 241 pp. 


\bibitem{Lin_100} CS. Lin and WM. NI, \emph{A counterexample to the nodal domain conjecture and a related
semilinear equation}, Proc. Amer. Mat. Soc., Vol. 102, 1988, pp. 271-277.





\bibitem{cordes_1} A. Maugeri, DK. Palagachev and   LG. Softova, \emph{Elliptic and Parabolic Equations with Discontinuous Coefficients} Wiley, Berlin (2000)



\bibitem{MP} R. Mazzeo and  F. Pacard. \emph{A construction of singular solutions for a semilinear elliptic equation using asymptotic analysis}, J. Diff. Geom. 44 (1996) 331-370.










    \bibitem{Passaseo} D.Passaseo,   \emph{ Nonexistence results for elliptic problems with supercritical nonlinearity
in nontrivial domains.} J. Funct. Anal. 114(1):97–105.(1993).




\bibitem{POHO} S.  Pohozaev, S. (1965). \emph{ Eigenfunctions of the equation $\Delta u + \lambda f(u)=0$.} Soviet. Math. Dokl. 6:1408–1411.








\bibitem{STRUWE}  Struwe, M. (1990). Variational Methods – Applications to Nonlinear Partial Differential Equations and Hamiltonian Systems. Berlin: Springer-Verlag.




\bibitem{cordes_3} G. Talenti, \emph{Sopra una classe di equazioni ellittiche a coefficienti misurabili}, Ann. Mat. Pura Appl. 69, 285-304 (1965)

\bibitem{cordes_4} G. Talenti, \emph{Equazioni lineari ellittiche in due variabili}, Matematiche 21, 339-376 (1966)











        
        \bibitem{addd_104} K Wang and J Wei, \emph{Analysis of Blow-up Locus and Existence of Weak Solutions for Nonlinear Supercritical Problems}, International Mathematics Research Notices, Volume 2015, Issue 10, 1 January 2015, Pages 2634-2670.
        












\end{thebibliography}
\end{document}